\documentclass[11pt,letterpaper]{article}
\usepackage{fullpage,latexsym,verbatim,times}
\usepackage{amsthm,amsfonts,amsmath,amstext,amssymb,amsopn}
\usepackage{graphicx,epsfig}
\usepackage{xspace}
\usepackage{color}
\usepackage{tikz}
\usepackage{nicefrac}
\usepackage{enumerate}
\usepackage{caption}
\usepackage{subcaption}

\usepackage{float}

\usepackage{url}

\newenvironment{rev}{\color{black}}

\newenvironment{ghh}{\color{black}}

\newtheorem{openprob}{Open Problem}
\newtheorem{theorem}{Theorem}[section]

\newtheorem{lemma}[theorem]{Lemma}
\newtheorem{prop}[theorem]{Proposition}

\newtheorem{definition}[theorem]{Definition}
\newtheorem{remark}[theorem]{Remark}
\newtheorem{example}[theorem]{Example}

\newcommand {\bit}{\begin{itemize} \item}\newcommand {\eit}{\end{itemize}}
\newcommand {\ben}{\begin{enumerate} \item}\newcommand {\een}{\end{enumerate}}
\newcommand {\bena}{\begin{enumerate}\renewcommand{\labelenumi}{\alph{enumi}.}\item}

\newcommand {\beqn}{\begin{equation}}\newcommand {\eeqn}{\end{equation}}
\newcommand {\beqan}{\begin{eqnarray}}\newcommand {\eeqan}{\end{eqnarray}}
\newcommand {\beqa}{\begin{eqnarray*}}\newcommand {\eeqa}{\end{eqnarray*}}
\newcommand {\barr}{\begin{array}}\newcommand {\earr}{\end{array}}
\newcommand {\bat}{\begin{tabular}}\newcommand {\eat}{\end{tabular}}

 \normalsize

\begin{document}

\title{Polynomial Norms}

\author{
Amir Ali Ahmadi\thanks{Amir Ali Ahmadi. ORFE, Princeton University, Sherrerd Hall, Princeton, NJ 08540, USA. Email: a\_a\_a@princeton.edu}    \footnotemark[4]  \and
Etienne de Klerk \thanks{Etienne de Klerk. Department Econometrics and Operations Research, TISEM, Tilburg University, 5000LE Tilburg, The Netherlands. Email: e.deklerk@uvt.nl} \and Georgina Hall \thanks{Georgina Hall, \textit{corresponding author.} ORFE, Princeton University, Sherrerd Hall, Princeton, NJ 08540, USA. Email: gh4@princeton.edu}  \thanks{Amir Ali Ahmadi and Georgina Hall are partially supported by the DARPA Young Faculty Award, the Young Investigator Award of the AFOSR, the CAREER Award of the NSF, the Google Faculty Award, and the Sloan Fellowship.}}

\date{}
%\thanks{Amir Ali Ahmadi and Georgina Hall are partially supported by the Young Investigator Award of the AFOSR, the CAREER Award of the NSF, the Google Faculty Award, and the Sloan Fellowship.}

\maketitle

\begin{abstract}
\noindent
In this paper, we study \emph{polynomial norms}, i.e. norms that are the $d^{\text{th}}$ root of a degree-$d$ homogeneous polynomial $f$. We first show that a necessary and sufficient condition for $f^{1/d}$ to be a norm is for $f$ to be strictly convex, or equivalently, convex and positive definite. Though not all norms come from $d^{\text{th}}$ roots of polynomials, we prove that any norm can be approximated arbitrarily well by a polynomial norm. We then investigate the computational problem of testing whether a form gives a polynomial norm. We show that this problem is strongly NP-hard already when the degree of the form is 4, but can always be answered {\ghh by solving a hierarchy of semidefinite programs}. We further study the problem of optimizing over the set of polynomial norms using semidefinite programming. To do this, we introduce the notion of \emph{r-sos-convexity} and extend a result of Reznick on sum of squares representation of positive definite forms to positive definite biforms. We conclude with some applications of polynomial norms to statistics and dynamical systems.

%We consider the following question: when is the $d^{th}$ root of a degree-$d$ form $f$ a norm? We present two different conditions, relating to convexity of $f$, under which the $d^{th}$ root of $f$ is a norm --- namely that $f$ has to be convex and positive definite, or equivalently, strictly convex. Any such norm will be termed a \emph{polynomial norm}. We first show that any norm can be approximated arbitrarily well by a polynomial norm. We then investigate the computational problem of testing whether a form gives a polynomial norm. We show that this problem is strongly NP-hard already when the degree of the form is 4, but can always be answered by testing feasibility of a semidefinite program (of possibly large size). We then study the problem of optimizing over the set of polynomial norms using semidefinite programming. To do this, we introduce the notion of \emph{r-sos-convexity} and extend a result of Reznick on sum of squares representation of positive definite forms to positive definite biforms. We conclude with some applications of polynomial norms to statistics and dynamical systems.
\end{abstract}

\small{\noindent \textit{Keywords:} polynomial norms, sum of squares polynomials, convex polynomials, semidefinite programming

}

%============================================================
\section{Introduction}
%============================================================

A function $f:\mathbb{R}^n \rightarrow \mathbb{R}$ is a \emph{norm} if it satisfies the following three properties:
\begin{enumerate}[(i)]
	\item positive definiteness: $f(x)>0, ~\forall x \neq 0,$ and $f(0)=0$.
	\item $1$-homogeneity: $f(\lambda x)=|\lambda| f(x),~ \forall x\in \mathbb{R}^n, ~\forall \lambda \in \mathbb{R}$.
	\item triangle inequality: $f(x+y)\leq f(x)+f(y), ~\forall x,y \in \mathbb{R}^n.$
\end{enumerate}
Some well-known examples of norms include the $1$-norm, $f(x)=\sum_{i=1}^n |x_i|$, the $2$-norm, $f(x)=\sqrt{\sum_{i=1}^n x_i^2}$, and the $\infty$-norm, $f(x)=\max_{i} |x_i|.$ Our focus throughout this paper is on norms that can be derived from multivariate polynomials. More specifically, we are interested in establishing conditions under which the $d^{th}$ root of a homogeneous polynomial of degree $d$ is a norm, where $d$ is an even number. We refer to the norm obtained when these conditions are met as \emph{a polynomial norm}. It is straightforward to see why we restrict ourselves to $d^{th}$ roots of degree-$d$ homogeneous polynomials. Indeed, nonhomogeneous polynomials cannot hope to satisfy the homogeneity condition of a norm and homogeneous polynomials of degree $d>1$ are not 1-homogeneous unless we take their $d^{th}$ root. The question of when the square root of a homogeneous quadratic polynomial is a norm (i.e., when $d=2$) has a well-known answer (see, e.g., \cite[Appendix A]{BoydBook}): a function $f(x)=\sqrt{x^TQx}$ is a norm if and only if the symmetric $n \times n$ matrix $Q$ is positive definite. In the particular case where $Q$ is the identity matrix, one recovers the $2$-norm. Positive definiteness of $Q$ can be checked in polynomial time using for example Sylvester's criterion (positivity of the $n$ leading principal minors of $Q$). This means that testing whether the square root of a quadratic form is a norm can be done in polynomial time. A similar characterization in terms of conditions on the coefficients are not known for polynomial norms generated by forms of degree greater than 2. In particular, it is not known whether one can efficiently test membership or optimize over the set of polynomial norms.

\paragraph{Outline and contributions.} In this paper, we study polynomial norms from a computational perspective. {\rev In Section \ref{sec:eq.charac.comp}, we give two different necessary and sufficient conditions under which the $d^{th}$ root of a degree-$d$ form $f$ will be a polynomial norm: namely, that $f$ be strictly convex, or (equivalently) that $f$ be convex and postive definite (Theorem \ref{th:norm.str.conv})}. Section \ref{sec:approx.norms} investigates the relationship between general norms and polynomial norms: while many norms are polynomial norms (including all $p$-norms with $p$ even), some norms are not (consider, e.g., the $1$-norm). We show, however, that any norm can be approximated to arbitrary precision by a polynomial norm (Theorem \ref{th:approx.poly.norm.sphere}). {\ghh While it is well known that polynomials can approximate continuous functions on compact sets arbitrarily well, the approximation result here needs to preserve the convexity and homogeneity properties of the original norm, and hence does not follow, e.g., from the Stone-Weierstrass theorem.} In Section \ref{sec:sos.approx}, we move on to complexity results and show that simply testing whether the $4^{th}$ root of a quartic form is a norm is strongly NP-hard (Theorem \ref{th:NP.hard}). {\ghh We then provide a semidefinite programming-based hierarchy for certifying that the $d^{th}$ root of a degree $d$ form is a norm (Theorem \ref{th:test.poly.norm}) and for optimizing over a subset of the set of polynomial norms (Theorem \ref{th:opt.poly.norms}).} The latter is done by introducing the concept of $r$-sum of squares-convexity (see Definition \ref{def:r.sos.convex}). We show that any form with a positive definite Hessian is $r$-sos-convex for some value of $r$, and present a lower bound on that value (Theorem \ref{th:r.sos.convex}). We also show that the level $r$ of the semidefinite programming hierarchy cannot be bounded as a function of the number of variables and the degree only (Theorem \ref{th:unif.bnd}). Finally, we cover some applications of polynomial norms in statistics and dynamical systems in Section \ref{sec:apps}. In Section \ref{sec:norm.reg}, we compute approximations of two different types of norms, polytopic gauge norms and $p$-norms with $p$ noneven, using polynomial norms. The techniques described in this section can be applied to norm regression. In Section \ref{sec:JSR.comp}, we use polynomial norms to prove stability of a switched linear system, a task which is equivalent to computing an upperbound on the joint spectral radius of a family of matrices.

%==============================================
\section{Two equivalent characterizations of polynomial norms} \label{sec:eq.charac.comp}
%==================================================
We start this section with {\rev a theorem} that provides conditions under which the $d^{th}$ root of a degree-$d$ form is a norm.
{\ghh  We will use this theorem in Section \ref{sec:sos.approx} to establish
	semidefinite programming-based approximations of polynomial norms. We remark that this result is generally assumed to be known by the optimization community. Indeed, some prior work on polynomial norms has been done by Dmitriev and Reznick in \cite{dmitriev1973structure,dmitriev1991extreme,reznick1979banach,Blenders_Reznick}. For completeness of presentation, however, and as we could not find the exact statement of this result in the form we present, we include it here with alternative proofs.
%These results are of the `folklore-type', but we include proofs for completeness of presentation, and since we could not findtheir explicit statements in the literature.
Throughout this paper, we suppose that the number of variables $n$ is larger or equal than $2$ and that $d$ is a positive even integer.}

{\rev

\begin{theorem}\label{th:norm.str.conv}
Let $f$ be a form of degree $d$ in $n$ variables. The following statements are equivalent: 
\begin{enumerate}[(i)]
	\item The function $f^{\frac{1}{d}}$ is a norm on $\mathbb{R}^n$. 
	\item The function $f$ is convex and positive definite. 
	\item The function $f$ is strictly convex,  i.e., $$f(\lambda x+ (1-\lambda)y)<\lambda f(x)+(1-\lambda)f(y),~\forall x\neq y,~\forall \lambda \in (0,1).$$
\end{enumerate} 
\end{theorem}
}

\begin{proof}
	$(i)\Rightarrow (ii)$
		If $f^{1/d}$ is a norm, then $f^{1/d}$ is positive definite, and so is $f$. Furthermore, any norm is convex and the $d^{th}$ power of a nonnegative convex function remains convex.
		
		$(ii) \Rightarrow (iii)$ Suppose that $f$ is convex, positive definite, but not strictly convex, i.e., there exists $\bar{x},\bar{y} \in \mathbb{R}^n$ with $\bar{x}\neq \bar{y}$, and $\gamma \in (0,1)$ such that $$f\left(\gamma \bar{x}+(1-\gamma)\bar{y} \right)=\gamma f(\bar{x})+ (1-\gamma) f(\bar{y}).$$ Let $g(\alpha)\mathrel{\mathop{:}}=f(\bar{x}+\alpha(\bar{y}-\bar{x})).$ Note that $g$ is a restriction of $f$ to a line and, consequently, $g$ is a convex, positive definite, univariate polynomial in $\alpha$. We now define \begin{align}\label{eq:expression.g}
		h(\alpha)\mathrel{\mathop{:}}=g(\alpha)-(g(1)-g(0))\alpha-g(0).
			\end{align} Similarly to $g$, $h$ is a convex univariate polynomial as it is the sum of two convex univariate polynomials. We also know that $h(\alpha)\geq 0, \forall \alpha \in (0,1)$. Indeed, by convexity of $g$, we have that $g(\alpha x+(1-\alpha)y)\geq \alpha g(x)+(1-\alpha)g(y), \forall x,y \in \mathbb{R}$ and $\alpha \in (0,1)$. This inequality holds in particular for $x=1$ and $y=0$, which proves the claim. Observe now that $h(0)=h(1)=0$. By convexity of $h$ and its nonnegativity over $(0,1)$, we have that $h(\alpha)=0$ on $(0,1)$ which further implies that $h=0$. Hence, from (\ref{eq:expression.g}), $g$ is an affine function. As $g$ is positive definite, it cannot be that $g$ has a nonzero slope, so $g$ has to be a constant. But this contradicts that $\lim_{\alpha \rightarrow \infty} g(\alpha) =\infty.$ To see why this limit must be infinite,  we show that $\lim_{||x|| \rightarrow \infty} f(x)=\infty.$ As $\lim_{\alpha \rightarrow \infty} ||\bar{x}+\alpha(\bar{y}-\bar{x})||=\infty$ and $g(\alpha)=f(\bar{x}+\alpha(\bar{y}-\bar{x}))$, this implies that $\lim_{\alpha \rightarrow \infty} g(\alpha)=\infty.$ To show that $\lim_{||x|| \rightarrow \infty} f(x)=\infty$, let $$x^*=\underset{||x||=1}{\text{argmin }} f(x).$$ By positive definiteness of $f$, $f(x^*)>0.$ Let $M$ be any positive scalar and define $R\mathrel{\mathop{:}}=(M/f(x^*))^{1/d}$. Then for any $x$ such that $||x||=R$, we have $$f(x)\geq \min_{||x||=R} f(x) \geq R^d f(x^*)=M,$$ where the second inequality holds by homogeneity of $f.$ Thus $\lim_{||x|| \rightarrow \infty} f(x)=\infty$.
		
		$(iii) \Rightarrow (i)$ Homogeneity of $f^{1/d}$ is immediate. Positivity follows from the first-order characterization of strict convexity: $$f(y)>f(x)+\nabla f(x)^T(y-x), ~\forall y\neq x.$$ Indeed, for $x=0$, the inequality becomes $f(y)>0, ~\forall y\neq 0 $, as $f(0)=0$ and $\nabla f(0)=0$. Hence, $f$ is positive definite, and so is $f^{1/d}$. It remains to prove the triangle inequality. Let $g\mathrel{\mathop{:}}=f^{1/d}$. Denote by $S_f$ and $S_g$ the 1-sublevel sets of $f$ and $g$ respectively. It is clear that $$S_g=\{x~|~f^{1/d}(x)\leq 1\}=\{x~|~f(x)\leq 1\}=S_f,$$
			and as $f$ is strictly convex (and hence quasi-convex), $S_f$ is convex and so is $S_g$. Let $x,y \in \mathbb{R}^n.$ We have that $\frac{x}{g(x)} \in S_g$ and $\frac{y}{g(y)} \in S_g$. From convexity of $S_g$,
			$$g \left( \frac{g(x)}{g(x)+g(y)} \cdot \frac{x}{g(x)}+\frac{g(y)}{g(x)+g(y)} \cdot \frac{y}{g(y)}  \right) \leq 1.$$
			Homogeneity of $g$ then gives us $$\frac{1}{g(x)+g(y)}g(x+y)\leq 1$$ which shows that triangle inequality holds.
\end{proof}

\section{Approximating norms by polynomial norms}  \label{sec:approx.norms}
%===================================================

It is easy to see that not all norms are polynomial norms. For example,
 the 1-norm $||x||_1=\sum_{i=1}^n |x_i|$ is not a polynomial norm. Indeed,
  all polynomial norms are differentiable at all but one point (the origin) whereas
  the 1-norm is nondifferentiable whenever one of the components of $x$ is equal to zero. In this section,
  we show that, though not every norm is a polynomial norm, any norm can be approximated to
   arbitrary precision by a polynomial norm (Theorem \ref{th:approx.poly.norm.sphere}).
%    The proof of this theorem is inspired from a proof by Ahmadi and Jungers in \cite{sosconvex_Lyap_cdc,Ahmadi_Jungers}.
     A related result is given by Barvinok in \cite{barvinok2003}. In that paper, he shows that any norm can be approximated by the $d$-th root of a nonnegative
      degree-$d$ form, and quantifies the quality of the approximation as a function of $n$ and $d$.
      The form he obtains however is not shown to be convex. In fact, in a
       later work \cite[Section 2.4]{barvinok2006computational}, Barvinok points out that it would be an interesting question
       to know whether any norm can be approximated by the $d^{th}$ root of a convex form with the same quality of approximation as for $d$-th roots of nonnegative forms.
        The result below is a step in that direction, although the quality of approximation is weaker than that by Barvinok \cite{barvinok2003}. {\rev We note that the form in Barvinok's construction is a sum of squares of other forms. Such forms are not necessarily convex. By contrast, the form that we construct is a sum of powers of linear forms and hence always convex. }

\begin{theorem} \label{th:approx.poly.norm.sphere}
{\ghh 	Let $|| \cdot||$ be any norm on $\mathbb{R}^n$. Then,
\footnote{We would like to thank an anonymous referee for suggesting the proof of part (i) of this theorem. We were previously showing that for any norm $|| \cdot||$ on $\mathbb{R}^n$
 and for any $\epsilon>0$, there exist an even integer $d$ and an $n$-variate positive definite form $f_d$ of degree $d$, which is a sum of powers of linear forms, and such that
	 	\begin{align*}
	 	(1-\epsilon)||x|| \leq f_d^{1/d}(x) \leq ||x||,~ \forall x \in \mathbb{R}^n.
	 	\end{align*}}
 for any even integer}
 $d \ge 2$:
 \begin{enumerate}[(i)]
 \item
 There exists an $n$-variate convex positive definite form $f_d$ of degree $d$ such that
	\begin{align}\label{eq:norm.approx.level.set}
	\frac{d}{n+d}\left(\frac{n}{n+d}\right)^{n/d}||x|| \leq f_d^{1/d}(x) \leq  ||x||, \quad~ \forall x \in \mathbb{R}^n.
	\end{align}
In particular, for any sequence $\{f_d\}$ $(d= 2,4,6,\ldots)$ of such polynomials one has
\[
\lim_{d\rightarrow \infty} {f_d^{1/d}(x) \over ||x||} = 1 \quad~ \forall x \in \mathbb{R}^n.
\]
\item
One may assume without loss of generality that $f_d$ in (i) is a nonnegative sum of $d^{\text{th}}$ powers of linear forms.
\end{enumerate}
\end{theorem}

\begin{proof}[Proof of (i)]
% proof construction by Referee 3 of FoCM
Fix any norm $\|\cdot \|$  on $\mathbb{R}^n$. We denote the Euclidean inner product  on $\mathbb{R}^n$ by $\langle \cdot,\cdot\rangle$, and
the unit ball with respect to $\|\cdot \|$ by
\[
B = \left\{x \in \mathbb{R}^n \; | \; \|x\| \le 1 \right\}.
\]
We denote the polar of $B$ with respect to $\langle \cdot,\cdot\rangle$ by
\[
B^\circ = \left\{y \in \mathbb{R}^n \; | \; \langle x,y\rangle  \le 1 \; \forall x \in B\right\}.
\]
Recall that $B^\circ$ is symmetric around the origin, because $B$ is.
One may express the given norm in terms of the polar as follows (see e.g.\ relation (3.1) in \cite{barvinok2003}):
\begin{equation}\label{norm ito polar}
\|x\| = \max_{y \in B^\circ} \langle x,y\rangle = \max_{y \in B^\circ} |\langle x,y\rangle| \quad \forall x \in \mathbb{R}^n.
\end{equation}
For given, even integer $d$, we define the polynomial
\begin{equation}\label{def:f_d}
  f_d(x) = \frac{1}{\mbox{vol} B^\circ} \int_{B^\circ} \langle x,y\rangle^d dy.
\end{equation}
Note that $f_d$ is indeed a convex form of degree $d$. In fact, we will show later on that $f_d$ may in fact be written as a nonnegative sum
of $d$th powers of linear forms.

By \eqref{norm ito polar}, one has
\[
f_d^{1/d}(x) \le \|x\| \quad~ \forall x \in \mathbb{R}^n.
\]
Now fix $x_0 \in \mathbb{R}^n$ such that $\|x_0\| = 1$. By \eqref{norm ito polar},
there exists a $y_0 \in B^\circ$ so that $\langle x_0,y_0\rangle = 1$.
Define the half-space
\[
H_+ = \left\{y \in \mathbb{R}^n \; | \; \langle x_0,y\rangle \ge 0  \right\}.
\]
Then, by symmetry,
\[
\mbox{vol}\left( H_+ \cap B^\circ\right) = \frac{1}{2} \mbox{vol}\left( B^\circ\right).
\]
For any $\alpha \in (0,1)$ we now define
\[
A_+(\alpha) = \left\{(1-\alpha)y + \alpha y_0 \; | \; y \in H_+\cap B^\circ \right\}.
\]
Then $A_+(\alpha) \subset B^\circ $, and
\[
\mbox{vol} A_+(\alpha) = \frac{1}{2} (1-\alpha)^n \mbox{vol}\left(B^\circ\right).
\]
Moreover
\begin{equation}
\label{ineq:Aplus}
\langle x_0,y\rangle \ge \alpha \quad \forall y \in A_+(\alpha),
\end{equation}
and
\begin{equation}
\label{eq:zero vol}
\mbox{vol} \left(A_+(\alpha)\cap A_{-}(\alpha)\right)=0.
\end{equation}
Letting $A_{-}(\alpha) = - A_+(\alpha)$, by symmetry one has $A_{-}(\alpha) \subset B^\circ$,
and
\begin{equation}
\label{ineq:Amin}
\langle x_0,y \rangle \le -\alpha \quad \forall y \in A_{-}(\alpha).
\end{equation}
Thus
\begin{eqnarray*}
  f^{1/d}(x_0)  &=& \left(\frac{1}{\mbox{vol} B^\circ} \int_{B^\circ} \langle x_0,y\rangle^d dy\right)^{1/d} \\
   &\ge &  \left(\frac{1}{\mbox{vol} B^\circ} \int_{A_+(\alpha)\cup A_{-}(\alpha)} \langle x_0,y\rangle^d dy\right)^{1/d}  \\
   &\ge & \left(\frac{\mbox{vol} A_+(\alpha)+\mbox{vol} A_-(\alpha)}{\mbox{vol} B^\circ} \alpha^d\right)^{1/d} \quad \mbox{(by \eqref{ineq:Aplus}, \eqref{ineq:Amin}, and \eqref{eq:zero vol})}\\
   &=& \alpha(1-\alpha)^{n/d}.
\end{eqnarray*}
The last expression is maximized by $\alpha = \frac{d}{n+d}$, yielding the leftmost inequality in \eqref{eq:norm.approx.level.set} in the statement of the theorem.
Finally, note that,
\[
  \lim_{d \rightarrow \infty} \frac{d}{n+d}\left(\frac{n}{n+d}\right)^{n/d}  = \lim_{t \downarrow 0} t^t (1+t)^{-(1+t)} = 1,\\
  \]
as required.
\end{proof}

For the proof of the second part of the theorem, we need a result concerning finite moment sequences of signed measures, given as the next lemma. {\ghh

\begin{lemma}[ \cite{rogosinski}, see Lemma 3.1 in \cite{Shapiro2001} for a simple proof]
\label{lemma:Shapiro}
Let $\Omega \subset \mathbb{R}^n$ be Lebesgue-measurable, and let $\mu$ be the normalized Lebesgue measure on $\Omega$, i.e.\ $\mu(\Omega) = 1$.
Denote the moments of $\mu$ by
\begin{equation}\label{def:moment alpha}
  m_\mu(\alpha) = \int_\Omega x^\alpha d\mu(x) \quad \forall \alpha \in \mathbb{N}_0^n,
\end{equation}
where $x^\alpha := \prod_i x_i^{\alpha_i}$ if $x = (x_1,\ldots,x_n)$ and $\alpha = (\alpha_1,\ldots,\alpha_n)$.
Let $S \subset \mathbb{N}_0^n$ be finite. Then there exists an atomic probability measure, say $\mu'$, supported on at most $|S|$ points in $\Omega$, such that
\[
m_\mu(\alpha) = m_{\mu'}(\alpha) \quad \forall \alpha \in S.
\]
\end{lemma}

We may now prove part 2 of Theorem \ref{th:approx.poly.norm.sphere}.

\begin{proof}[Proof of (ii) of Theorem \ref{th:approx.poly.norm.sphere}]
Let $\mu$ be the normalized Lebesgue measure on $B^\circ$, i.e.\
\[
d\mu(y) = \frac{1}{\mbox{vol} B^\circ} dy.
\]
We now use the multinomial theorem
\[
\langle x,y\rangle^{d} = \sum_{|\alpha| = d} {d \choose \alpha} x^\alpha y^\alpha,
\]
with the multinomial notation
\[
|\alpha| = \sum_i \alpha _i, \; {d \choose \alpha} = \frac{d!}{\alpha_1!\cdots\alpha_n!},
\]
to  rewrite $f_d$ in \eqref{def:f_d} as
\begin{equation}
\label{pol_f}
f_d(x) = \sum_{|\alpha| = d} {d \choose \alpha} m_\mu(\alpha)x^\alpha,
\end{equation}
where $m_\mu(\alpha)$ is the moment of order $\alpha$ of $\mu$, as defined in \eqref{def:moment alpha}.

By Lemma \ref{lemma:Shapiro}, there exist $\bar y^{(1)},\ldots, \bar y^{(p)} \in B^\circ$ with
$p = {d+n-1 \choose d}$ so that
\[
m_\mu(\alpha) = \sum_{j=1}^p \lambda_j (\bar y^{(j)})^\alpha,
\]
for some $\lambda_j \ge 0$ $(j= 1,\ldots,p)$ with $\sum_{j=1}^p \lambda_j =1$. Substituting in \eqref{pol_f}, one has
\begin{eqnarray*}
f_d(x) &=& \sum_{|\alpha| = d} {d \choose \alpha} \sum_{j=1}^p \lambda_j (\bar y^{(j)})^\alpha x^\alpha \\
  &=& \sum_{j=1}^p \lambda_j \left( \sum_{|\alpha| = d} {d \choose \alpha} (\bar y^{(j)})^\alpha x^\alpha \right) \\
  &=& \sum_{j=1}^p \lambda_j \langle \bar y^{(j)}, x\rangle^{d},
\end{eqnarray*}
as required.
\end{proof}
}

%\begin{remark}
%	We remark that the polynomial norm constructed in Theorem \ref{th:approx.poly.norm.sphere} is
%the $d^{th}$-root of an \emph{sos-convex} polynomial.
% Hence, one can approximate any norm on $\mathbb{R}^n$ by searching for a polynomial norm using semidefinite programming.
%  To see why the polynomial $f$ in (\ref{eq:poly.def}) is sos-convex, observe that linear forms are sos-convex and that an even power of an sos-convex form is sos-convex.
%\end{remark}

%=================================================
\section{Semidefinite programming-based approximations of polynomial norms} \label{sec:sos.approx}
%====================================================

\subsection{Complexity}
It is natural to ask whether testing if the $d^{th}$ root of a given degree-$d$ form is a norm can be done in polynomial time. In the next theorem, we show that, unless $P=NP$, this is not the case even when $d=4$.
\begin{theorem}\label{th:NP.hard}
	Deciding whether the $4^{th}$ root of a quartic form is a norm is strongly NP-hard.
\end{theorem}
\begin{proof}
	The proof of this result is adapted from a proof in \cite{NPhard_Convexity_MathProg}. Recall that the CLIQUE problem can be described thus: given a graph $G=(V,E)$ and a positive integer $k$, decide whether $G$ contains a clique of size at least $k$. The CLIQUE problem is known to be NP-hard \cite{GareyJohnson_Book}. We will give a reduction from CLIQUE to the problem of testing convexity and positive definiteness of a quartic form. The result then follows from Theorem \ref{th:norm.str.conv}. Let $\omega(G)$ be the clique number of the graph at hand, i.e., the number of vertices in a maximum clique of $G$. Consider the following quartic form $$b(x;y)\mathrel{\mathop{:}}=-2k\sum_{i,j \in E} x_ix_jy_iy_j-(1-k)\left(\sum_i x_i^2\right)\left(\sum_i y_i^2\right).$$ In \cite{NPhard_Convexity_MathProg}, using in part a result in \cite{Ling_et_al_Biquadratic}, it is shown that
	\begin{align} \label{eq:clique.str.conv}
	&\omega(G)\leq k \Leftrightarrow b(x;y)+\frac{n^2 \gamma}{2} \left( \sum_{i=1}^n x_i^4+\sum_{i=1}^n y_i^4 +\sum_{1\leq i<j \leq n} (x_i^2x_j^2+ y_i^2y_j^2)\right)
	\end{align}
	is convex and $b(x;y)$ is positive semidefinite. Here, $\gamma$ is a positive constant defined as the largest coefficient in absolute value of any monomial present in some entry of the matrix $\left[\frac{\partial^2 b(x;y)}{\partial x_i \partial y_j}\right]_{i,j}$. As $\sum_i x_i^4+\sum_i y_i^4$ is positive definite and as we are adding this term to a positive semidefinite expression, the resulting polynomial is positive definite. Hence, the equivalence holds if and only if the quartic on the righthandside of the equivalence in (\ref{eq:clique.str.conv}) is convex and positive definite.
\end{proof}

Note that this also shows that strict convexity is hard to test for quartic forms (this is a consequence of Theorem \ref{th:norm.str.conv}). A related result is Proposition 3.5. in \cite{NPhard_Convexity_MathProg}, which shows that testing strict convexity of a polynomial of even degree $d\geq 4$ is hard. However, this result is not shown there for \emph{forms}, hence the relevance of the previous theorem.

 {\ghh Theorem \ref{th:NP.hard} rules out the possibility of a pseudo-polynomial time characterization of polynomial norms (unless P=NP) and motivates the study of tractable sufficient conditions.
 %Theorem \ref{th:NP.hard} motivates the study of tractable sufficient conditions to be a polynomial norm.
The sufficient conditions we consider next are based on semidefinite programming. Semidefinite programs can be solved to arbitrary accuracy in polynomial time~\cite{VaB:96} and technology for solving this class of problems is rapidly improving~\cite{mosek,nie_large_scale,qsdpnal,sun_large_scale,Symmetry_groups_Gatermann_Pablo,riener_symmetry,iSOS_journal}.}
%We now present a series of sufficient conditions under which $d^{th}$ roots of degree-$d$ forms become norms. Contrary to the conditions described in Section \ref{sec:eq.charac.comp}, these conditions are computationally tractable; we show in particular that they can be reformulated as semidefinite programs. However, as one could expect due to differences in computational difficulty, the conditions we present are not necessary. We formally quantify how much is lost from inner approximating the set of polynomial norms by semidefinite programming-based approximations.

\subsection{Sum of squares polynomials and semidefinite programming review} \label{sec:sos.review}

We start this section by reviewing the notion of \emph{sum of squares polynomials} and related concepts such as \emph{sum of squares-convexity}. We say that a polynomial $f$ is a \emph{sum of squares} (sos) if $f(x)=\sum_i q_i^2(x)$, for some polynomials $q_i$. Being a sum of squares is a sufficient condition for being nonnegative. The converse however is not true, as is exemplified by the Motzkin polynomial
\begin{align}\label{eq:motzkin}
M(x,y)=x^4y^2 + x^2y^4 − 3x^2y^2 + 1
\end{align}
which is nonnegative but not a sum of squares \cite{MotzkinSOS}.
The sum of squares condition is a popular surrogate for nonnegativity due to its tractability. Indeed, while testing nonnegativity of a polynomial of degree greater or equal to 4 is a hard problem, testing whether a polynomial is a sum of squares can be done using \emph{semidefinite programming.} This comes from the fact that a polynomial $p$ of degree $d$ is a sum of squares if and only if there exists a positive semidefinite matrix $Q$ such that $f(x)=z(x)^TQz(x)$, where $z(x)$ is the standard vector of monomials of degree up to $d$ (see, e.g., \cite{PabloPhD}). As a consequence, any optimization problem over the coefficients of a set of polynomials which includes a combination of affine constraints and sos constraints on these polynomials, together with a linear objective can be recast as a semidefinite program. These type of optimization problems are known as \emph{sos programs} and have found widespread applications in recent years \cite{sdprelax,lasserre_moment,henrion2004solving,henrion2005}.

Though not all nonnegative polynomials can be written as sums of squares, the following theorem by Artin \cite{Artin_Hilbert17} circumvents this problem using sos multipliers.
\begin{theorem}[Artin \cite{Artin_Hilbert17}]\label{th:artin}
	For any nonnegative polynomial $f$, there exists an sos polynomial $q$ such that $q \cdot f$ is sos.
\end{theorem}
This theorem in particular implies that if we are given a polynomial $f$, then we can always check its nonnegativity using an sos program that searches for $q$ (of a fixed degree). However, this result does not allow us to optimize over the set of nonnegative polynomials or positive semidefinite polynomial matrices using an sos program (as far as we know). This is because, in that setting, products of decision varibles arise from multiplying polynomials $f$ and $q$, whose coefficients are decision variables.
%This theorem in particular implies that if we fix the degree of the multiplier $q$ and if we are given a fixed polynomial $f$, nonnegativity of $f$ can be tested using sos programs which will search for $q$. By progressively pushing up the degree of $q$, one can hope to recover a certificate of nonnegativity of $f$. If $f$ is not given though, both the coefficients of $f$ and the coefficients of $q$ are decision variables and the problem becomes nonconvex.

By adding further assumptions on $f$, Reznick showed in \cite{Reznick_Unif_denominator} that
one could further pick $q$ to be a power of $\sum_i x_i^2$ {\rev (we refer the reader to \cite[Chapter 1, Section 3]{barvinok2002course} for another nice presentation of Reznick's result)}. In what follows, $S^{n-1}$ denotes the unit sphere in $\mathbb{R}^n.$

\begin{theorem}[Reznick \cite{Reznick_Unif_denominator}]\label{th:reznick}
	Let $f$ be a positive definite form of degree $d$ in $n$ variables and define $$\epsilon(f)\mathrel{\mathop{:}}=\frac{\min\{f(u)~|~ u \in S^{n-1} \}}{\max\{f(u)~|~ u \in S^{n-1}\}}.$$ If $r\geq  \frac{nd(d-1)}{4 \log(2)\epsilon(f)}-\frac{n+d}{2}$, then $(\sum_{i=1}^n x_i^2)^r \cdot f$ is a sum of squares.
\end{theorem}
Motivated by this theorem, the notion of $r$-sos polynomials can be defined: a polynomial $f$ is said to be $r$-sos if $(\sum_i x_i^2)^r \cdot f$ is sos. Note that it is clear that any $r$-sos polynomial is nonnegative and that the set of $r$-sos polynomials is included in the set of $(r+1)$-sos polynomials. The Motzkin polynomial in (\ref{eq:motzkin}) for example is $1$-sos although not sos.

To end our review, we briefly touch upon the concept of sum of squares-convexity (sos-convexity), which we will build upon in the rest of the section. Let $H_f$ denote the Hessian matrix of a polynomial $f$. We say that $f$ is \emph{sos-convex} if $y^TH_f(x)y$ is a sum of squares (as a polynomial in $x$ and $y$). As before, optimizing over the set of sos-convex polynomials can be cast as a semidefinite program. Sum of squares-convexity is obviously a sufficient condition for convexity via the second-order characterization of convexity. However, there are convex polynomials which are not sos-convex (see, e.g., \cite{AAA_PP_not_sos_convex_journal}). For a more detailed overview of sos-convexity including equivalent characterizations and settings in which sos-convexity and convexity are equivalent, refer to \cite{AAA_PP_table_sos-convexity}.

\subsubsection{Notation}

Throughout, we will use the notation $H_{n,d}$ (resp. $P_{n,d}$) to denote the set of forms (resp. positive semidefinite, aka nonnegative, forms) in $n$ variables and of degree $d$. We will futhermore use the falling factorial notation $(t)_0=1$ and $(t)_k=t(t-1)\ldots(t-(k-1))$ for a positive integer $k$.
{\ghh

\subsection{Certifying validity of a polynomial norm} \label{sec:test}}

In this subsection, we assume that we are given a form $f$ of degree $d$ and we would like to {\ghh prove that} $f^{1/d}$ is a norm using semidefinite programming.

\begin{theorem} \label{th:test.poly.norm}
	Let $f$ be a degree-$d$ form. Then $f^{1/d}$ is a polynomial norm if and only if there exist $c>0$, $r \in \mathbb{N}$, {\ghh and an sos form $q(x)$ such that $q(x) \cdot y^T H_f(x) y$ is sos} and $\left(f(x)-c(\sum_i x_i^2)^{d/2}\right) (\sum_i x_i^2)^r$ is sos. Furthermore, this condition can be checked using semidefinite programming.
\end{theorem}

{\ghh
	To show this result, we require a counterpart to Theorem \ref{th:artin} for matrices, which we present below. {\rev We say that a polynomial matrix $H(x)$, i.e., a matrix whose entries are polynomials in $x=(x_1,\ldots,x_n)$, is positive semidefinite if it has nonnegative eigenvalues for all substitutions $x \in \mathbb{R}^n$. } 
	\begin{prop}[\cite{procesischaher, gondard, hillarnie}]\label{prop:mat.art}
	If $H(x)$ is a positive semidefinite polynomial matrix, then there exists a sum of squares polynomial $q(x)$ such that $q(x)\cdot y^TH(x)y$ is a sum of squares.
	\end{prop}
\begin{proof}
	This is an immediate consequence of a theorem by Procesi
	and Schacher \cite{procesischaher} and independently Gondard and Ribenboim \cite{gondard}, reproven by Hillar and Nie \cite{hillarnie}. This theorem states that if $H(x)$ is a symmetric polynomial matrix that is positive semidefinite for all $x \in \mathbb{R}^n,$ then
	\begin{align}
	H(x)=\sum_i A_i(x)^2, \nonumber
	\end{align} where the matrices $A_i(x)$ are symmetric and have rational functions as entries. Let $p(x)$ be the polynomial obtained by multiplying all denominators of the rational functions involved in any of the matrices $A_i$. Note that $p(x)^2\cdot y^T H(x)y$ is a sum of squares as $$p(x)^2\cdot y^TH(x)y=\sum_i p(x)^2 y^TA_i(x)^2y=\sum_i ||p(x)A_i(x)y||^2.$$ However, $p(x) \cdot A_i(x)$ is now a matrix with polynomial entries, which gives the result.
\end{proof}
	We remark that this result does not immediately follow from the theorem given by Artin as the multiplier $q$ does not depend on $x$ and $y$, but solely on $x$. We now prove Theorem \ref{th:test.poly.norm}.}

\begin{proof}[Proof of Theorem \ref{th:test.poly.norm}]
	It is immediate to see that if there exist such a $c$, $r$, and $q$, then $f$ is convex and positive definite. From Theorem \ref{th:norm.str.conv}, this means that $f^{1/d}$ is a polynomial norm.
	
	Conversely, if $f^{1/d}$ is a polynomial norm, then, by Theorem \ref{th:norm.str.conv}, $f$ is convex and positive definite. As $f$ is convex, the polynomial $y^TH_f(x)y$ is nonnegative. {\ghh Using Proposition \ref{prop:mat.art}, we conclude that there exists an sos polynomial $q(x)$ such that $q(x,y) \cdot y^TH_f(x)y$ is sos. } We now show that, as $f$ is positive definite, there exist $c>0$ and $r \in \mathbb{N}$ such that  $\left(f(x)-c(\sum_i x_i^2)^{d/2}\right) (\sum_i x_i^2)^r$ is sos. Let $f_{min}$ denote the minimum of $f$ on the sphere. As $f$ is positive definite, $f_{min}>0.$ We take $c\mathrel{\mathop{:}}=\frac{f_{min}}{2}$ and consider $g(x)\mathrel{\mathop{:}}=f(x)-c(\sum_i x_i^2)^{d/2}$. We have that $g$ is a positive definite form: indeed, if $x$ is a nonzero vector in $\mathbb{R}^n$, then $$\frac{g(x)}{||x||^d}=\frac{f(x)}{||x||^d}-c=f \left( \frac{x}{||x||}\right)-c>0, $$
	by homogeneity of $f$ and definition of $c$. Using Theorem \ref{th:reznick}, $\exists r \in \mathbb{N}$ such that $g(x)(\sum_i x_i^2)^r$ is sos.
	
	For fixed $r$, a given form $f$, and a fixed degree $d$, one can search for $c>0$ and an sos form $q$ of degree $d$ such that {\ghh $q(x) \cdot y^T H_f(x) y$} is sos and $\left(f(x)-c(\sum_i x_i^2)^{d/2}\right) (\sum_i x_i^2)^r$ is sos using semidefinite programming. This is done by solving the following semidefinite feasibility problem:
	\begin{equation}\label{eq:SDP.test.poly.norm}
	\begin{aligned}
	&{\ghh q(x)} \text{ sos}\\
	&c \geq 0 \\
	& {\ghh q(x)} \cdot y^TH_f(x)y \text{ sos}\\
	& \left(f(x)-c \left(\sum_i x_i^2\right)^{d/2}\right) \left(\sum_i x_i^2\right)^r \text{ sos},
	\end{aligned}
\end{equation}
	where the unknowns are the coefficients of $q$ and the real number $c$.
\end{proof}

\begin{remark}
	We remark that we are not imposing $c>0$ in the semidefinite program above. This is because, in practice, especially if the semidefinite program is solved with interior point methods, the solution returned by the solver will be in the interior of the feasible set, and hence $c$ will automatically be positive.
	One can slightly modify (\ref{eq:SDP.test.poly.norm}) however to take the constraint $c>0$ into consideration {\rev explicitly}. Indeed, consider the following semidefinite feasibility problem where both the degree of $q$ and the integer $r$ are fixed:
	\begin{align}
	&{\ghh q(x)} \text{ sos} \nonumber\\
	& \gamma \geq 0 \nonumber \\
	& {\ghh q(x)} \cdot y^TH_f(x)y \text{ sos} \nonumber \\
	& \left(\gamma f(x)-\left(\sum_i x_i^2\right)^{d/2}\right) \left(\sum_i x_i^2\right)^r \text{ sos}. \label{eq:remark4.5}\\ \nonumber
	\end{align}
 It is easy to check that (\ref{eq:remark4.5}) is feasible with $\gamma \geq 0$ if and only if the last constraint of (\ref{eq:SDP.test.poly.norm}) is feasible with $c>0$. To see this, take $c=1/\gamma$ and note that $\gamma$ can never be zero.
\end{remark}

To the best of our knowledge, we cannot use the approach described in Theorem \ref{th:test.poly.norm} to optimize over the set of polynomial norms with a semidefinite program. This is because of the product of decision variables in the coefficients of $f$ and $q$. The next subsection will address this issue.

\subsection{Optimizing over the set of polynomial norms}\label{sec:opt}

In this subsection, we consider the problem of optimizing over the set of polynomial norms. To do this, we introduce the concept of $r$-sos-convexity. Recall that the notation $H_f$ references the Hessian matrix of a form $f$.

\subsubsection{Positive definite biforms and r-sos-convexity}

\begin{definition}\label{def:r.sos.convex}
	For an integer $r$, we say that a polynomial $f$ is $r$-sos-convex if $y^TH_f(x)y \cdot (\sum_i x_i^2)^r$ is sos.
\end{definition}

Observe that, for fixed $r$, the property of $r$-sos-convexity can be checked using semidefinite programming
 (though the size of this SDP gets larger as $r$ increases). Any polynomial that is $r$-sos-convex is convex.
  Note that the set of $r$-sos-convex polynomials is a subset of the set of $(r+1)$-sos-convex polynomials and that the case $r=0$
  corresponds to the set of sos-convex polynomials. {\ghh We remark that $f_d$ in Theorem \ref{th:approx.poly.norm.sphere} is in fact sos-convex, since
  it is a sum of squares of linear forms.} Thus Theorem \ref{th:approx.poly.norm.sphere} implies that any norm on $\mathbb{R}^n$
may be approximated arbitrarily well by a polynomial norm that corresponds to a sos-convex form.

It is natural to ask whether any convex polynomial $f$ is $r$-sos-convex for some $r$. Our next theorem shows that this is the case {\rev provided that the biform $y^TH_f(x)y$, where $H_f(x)$ is the Hessian of $f$, is positive over the bi-sphere.}

\begin{theorem}\label{th:r.sos.convex}
	Let $f$ be a form of degree $d$ such that $y^TH_f(x)y > 0$ for $(x,y) \in S^{n-1} \times S^{n-1}$. Let $$\eta(f)\mathrel{\mathop{:}}= \frac{\min\{y^TH_f(x)y~|~ (x,y) \in S^{n-1} \times S^{n-1} \}}{\max\{y^TH_f(x)y~|~ (x,y) \in S^{n-1} \times S^{n-1}\}}.$$
If $r\geq  \frac{n(d-2)(d-3)}{4 \log(2)\eta(f)}-\frac{n+d-2}{2}-d$, then $f$ is $r$-sos-convex.
\end{theorem}

\begin{remark}
	Note that $\eta(f)$ can also be interpreted as $$\eta(f)  = \frac{\min_{x \in S^{n-1}} \lambda_{\min} (H_f(x)) }{\max_{x \in S^{n-1}} \lambda_{\max}(H_f(x))}=\frac{1}{ \max_{x \in S^{n-1}} \|H^{-1}_f(x)\|_2 \cdot \max_{x \in S^{n-1}} \|H_f(x)\|_2 }.$$

\end{remark}

\begin{remark}
Theorem \ref{th:r.sos.convex} is a generalization of Theorem \ref{th:reznick} by Reznick. Note though that this is not an immediate generalization. First, $y^TH_f(x)y$ is not a positive definite form (consider, e.g., $y=0$ and any nonzero $x$). Secondly, note that the multiplier is $(\sum_i x_i^2)^r$ and does not involve the $y$ variables. (As we will see in the proof, this is essentially because $y^TH_f(x)y$ is quadratic in $y$.) {\ghh It is not immediate that the multiplier should have this specific form. From Theorem \ref{th:reznick}, it may perhaps seem more natural that the multiplier be $(\sum_i x_i^2+\sum_iy_i^2)^r$. It turns out in fact that such a multiplier would not give us the correct property (contrarily to $(\sum_i x_i^2)^r$) as there exist forms $f$ whose Hessian is positive definite for all $x$ but for which the form $$y^TH_f(x)y (\sum_i x_i^2+\sum_i y_i^2)^r$$ is not sos for any $r$. For a specific example, consider the form $f$ in $3$ variables and of degree $8$ given in \cite[Theorem 3.2]{AAA_PP_not_sos_convex_journal}. It is shown in \cite{AAA_PP_not_sos_convex_journal} that (i) $f$ is convex; (ii) the $(1,1)$ entry of the Hessian of $f$, $H_f^{1,1}$, is not sos; and (iii) $(x_1^2+x_2^2+x_3^2)y^TH_f(x)y$ is sos.
		
Suppose for the sake of contradiction that $$q_r(x,y)\mathrel{\mathop{:}}=y^TH_f(x)y (\sum_i x_i^2+\sum_i y_i^2)^r$$ is sos for some $r$. This implies that the polynomial $$q_r(x,\alpha,0,0)=\alpha^2 H_p^{1,1} (x) \left( \sum_{i=1}^3 x_i^2+\alpha^2 \right)^r$$ should be sos for any $\alpha$ as it is obtained from $q_r$ by setting $y=(\alpha,0,0)^T$. Expanding this out, we get
	\begin{align*}
	q_r(x,\alpha,0,0)&= \alpha^2H_f^{1,1}(x) \sum_{k=0}^r \binom{r}{k} \left( \sum_{i=1}^3 x_i^2 \right)^{k}\alpha^{2r-2k}\\
	&= \alpha^{2r+2}H_f^{1,1}(x)+\sum_{k=1}^r  \binom{r}{k} \left( \sum_{i=1}^3 x_i^2 \right)^{k}\alpha^{2r-2k+2}H_f^{1,1}(x)\\
	&=\alpha^{2r+2} \left( H_f^{1,1}(x)+\sum_{k=1}^r  \binom{r}{k} \left( \sum_{i=1}^3 x_i^2 \right)^{k}\frac{1}{\alpha^{2k}}H_f^{1,1}(x) \right).
	\end{align*}
	From arguments (ii) and (iii) above, we know that $H_f^{1,1}(x)$ is not sos but $$\sum_{k=1}^r  \binom{r}{k} \left( \sum_{i=1}^3 x_i^2 \right)^{k}\frac{1}{\alpha^{2k}}H_f^{1,1}(x)$$ is sos. Fixing $\alpha$ large enough, we can ensure that $q_r(x,\alpha,0,0)$ is not sos. This contradicts our previous assumption.

}
\end{remark}

\begin{remark}
	Theorem \ref{th:r.sos.convex} can easily be adapted to biforms of the type $\sum_j f_j(x)g_j(y)$ where $f_j$'s are forms of degree $d$ in $x$ and $g_j$'s are forms of degree $\tilde{d}$ in $y$. In this case, there exist integers $s,r$ such that $$\sum_j f_j(x)g_j(y) \cdot (\sum_i x_i^2)^r \cdot (\sum_i y_i^2)^s$$ is sos. For the purposes of this paper however and the connection to polynomial norms, we will show the result in the particular case where the biform of interest is $y^TH_f(x)y.$
\end{remark}

We associate to any form $f \in H_{n,d}$, the $d$-th order differential operator $f(D)$, defined by replacing each occurence of $x_j$ with $\frac{\partial }{\partial x_j}$. For example, if $f(x_1,\ldots,x_n)\mathrel{\mathop{:}}=\sum_i c_i x_1^{a_1^i}\ldots x_{n}^{a_i^n}$ where $c_i \in \mathbb{R}$ and $a_{j}^i \in \mathbb{N}$, then its differential operator will be $$f(D)=\sum_{i} c_i \frac{\partial^{a_1^i}}{\partial x_1^{a_1^i}}\ldots \frac{\partial^{a_n^i}}{\partial x_n^{a_n^i}}.$$

Our proof will follow the structure of the proof of Theorem \ref{th:reznick} given in \cite{Reznick_Unif_denominator} and reutilize some of the results given in the paper which we quote here for clarity of exposition.

\begin{prop}[\cite{Reznick_Unif_denominator}, see Proposition 2.6]\label{prop:sq2lin}
	For any nonnegative integer $r$, there exist nonnegative rationals $\lambda_k$ and integers $\alpha_{kl}$ such that
	$$(x_1^2+\ldots+x_n^2)^r=\sum_k \lambda_k(\alpha_{k1}x_1+\ldots+\alpha_{kn}x_n)^{2r}.$$
\end{prop}
For simplicity of notation, we will let $\alpha_k\mathrel{\mathop{:}}=(\alpha_{k1},\ldots,\alpha_{kl})^T$ and $x\mathrel{\mathop{:}}=(x_1,\ldots,x_n)^T$. Hence, we will write $\sum_k \lambda_k(\alpha_k^Tx)^{2r}$ to mean $\sum_k \lambda_k (a_{k1}x_1+\ldots+a_{kn}x_n)^{2r}$.

\begin{prop}[\cite{Reznick_Unif_denominator}, see Proposition 2.8]\label{prop:diff2sum}
	If $g \in H_{n,e}$ and $h=\sum_k \lambda_k (\alpha_k^Tx)^{d+e} \in H_{n,d+e}$, then $$g(D)h=(d+e)_e \sum_k \lambda_k g(\alpha_k) (\alpha_k^Tx)^d.$$
\end{prop}

\begin{prop}[\cite{Reznick_Unif_denominator}, see {\rev Theorems 3.7 and 3.9}]\label{prop:def.Phi} For $f \in H_{n,d}$ and $s \geq d$, we define $\Phi_s(f) \in H_{n,d}$ by
	\begin{align}\label{eq:def.Phi}
	f(D) (x_1^2+\ldots+x_n^2)^s=\mathrel{\mathop{:}} \Phi_s(f) (x_1^2+\ldots+x_n^2)^{s-d}.
	\end{align}
The inverse $\Phi_s^{-1}(f)$ of $\Phi_s(f)$ exists and this is a map verifying $\Phi_s(\Phi_s^{-1}(f))=f.$
\end{prop}

\begin{prop}[\cite{Reznick_Unif_denominator}, see Theorem 3.12 ]\label{prop:pd.Phi}
	Suppose $f$ is a positive definite form in $n$ variables and of degree $d$ and let $$\epsilon(f)=\frac{\min\{f(u)~|~u \in S^{n-1}\}}{\max\{f(u)~|~u \in S^{n-1}\}}.$$ If $s\geq \frac{nd(d-1)}{4\log(2)\epsilon(f)}-\frac{n-d}{2}$, then $\Phi^{-1}_s(f) \in P_{n,d}.$
	
\end{prop}

We will focus throughout the proof on biforms of the following structure
\begin{align} \label{eq:def.F}
F(x;y)\mathrel{\mathop{:}}=\sum_{1\leq i,j \leq n} y_iy_jp_{ij}(x),
\end{align}
where $p_{ij}(x) \in H_{n,d}$, for all $i,j$, and some even integer $d$. Note that the polynomial $y^TH_f(x)y$ (where $f$ is some form) has this structure. We next present three lemmas which we will then build on to give the proof of Theorem \ref{th:r.sos.convex}.

\begin{lemma}\label{lem:F.sos}
	For a biform $F(x;y)$ of the structure in (\ref{eq:def.F}), define the operator $F(D;y)$ as $$F(D;y)=\sum_{ij} y_iy_jp_{ij}(D).$$ If $F(x;y)$ is positive semidefinite (i.e., $F(x;y)\geq 0,~ \forall x,y$), then, for any $s \geq 0$, the biform $$F(D;y)(x_1^2+\ldots+x_n^2)^s$$ is a sum of squares.
\end{lemma}
\begin{proof}
	Using Proposition \ref{prop:sq2lin}, we have $$(x_1^2+\ldots+x_n^2)^s=\sum_l \lambda_l(\alpha_{l1}x_1+\ldots \alpha_{ln}x_n)^{2s},$$ where $\lambda_l \geq 0$ and $\alpha_l \in \mathbb{Z}^n.$ Hence, applying Proposition \ref{prop:diff2sum}, we get
	\begin{align}
	F(D;y)(x_1^2+\ldots+x_n^2)^s &= \sum_{i,j} y_iy_j(p_{ij}(D)(x_1^2+\ldots+x_n^2)^s) \nonumber\\
	&= \sum_{i,j} y_iy_j \left( (2s)_d \sum_l \lambda_l p_{ij}(\alpha_l) (\alpha_l ^Tx)^{2s-d} \right) \nonumber\\
	&=(2s)_d \sum_l \lambda_l (\alpha_l ^Tx)^{2s-d} \sum_{i,j} y_iy_jp_{ij}(\alpha_l). \label{eq:proof.operator.sos}
	\end{align}
	Notice that $\sum_{i,j} y_iy_jp_{ij}(\alpha_l)$ is a quadratic form in $y$ which is positive semidefinite by assumption, which implies that it is a sum of squares (as a polynomial in $y$). Furthermore, as $\lambda_l \geq 0 ~\forall l$ and $(\alpha_l^Tx)^{2s-d}$ is an even power of a linear form, we have that $\lambda_l (\alpha_l^Tx)^{2s-d}$ is a sum of squares (as a polynomial in $x$). Combining both results, we get that (\ref{eq:proof.operator.sos}) is a sum of squares.
\end{proof}

We now extend the concept introduced by Reznick in Proposition \ref{prop:def.Phi} to biforms.
\begin{lemma} \label{lem:phi.def}
	For a biform $F(x;y)$ of the structure as in (\ref{eq:def.F}), we define the biform $\Psi_{s,x}(F(x;y))$ as
	\begin{align*}
	\Psi_{s,x}(F(x;y))\mathrel{\mathop{:}}=\sum_{i,j} y_iy_j \Phi_{s}(p_{ij}(x)),
	\end{align*}
where $\Phi_s$ is as in (\ref{eq:def.Phi}). Define
\begin{align*}
	\Psi_{s,x}^{-1}(F(x;y))\mathrel{\mathop{:}}=\sum_{i,j} y_iy_j \Phi_{s}^{-1}(p_{ij}(x)),
	\end{align*}
where $\Phi_s^{-1}$ is the inverse of $\Phi_s$. Then, we have
\begin{align}
F(D;y)(x_1^2+\ldots+x_n^2)^s=\Psi_{s,x} (F)(x_1^2+\ldots+x_n^2)^{s-d} \label{eq:def.Phi.s.x}
\end{align}
and
\begin{align}
\Psi_{s,x}(\Psi_{s,x}^{-1}(F))=F. \label{eq:Phi.s.x.inv}
\end{align}
\end{lemma}
\begin{proof}
We start by showing that (\ref{eq:def.Phi.s.x}) holds:
\begin{align}
F(D;y)(x_1^2+\ldots+x_n^2)^s &= \sum_{i,j}y_i y_jp_{ij}(D)(x_1^2+\ldots+x_n^2)^s \nonumber \\
&\underset{\text{using (\ref{eq:def.Phi})}}{=} \sum_{i,j} y_iy_j\Phi_s(p_{ij}(x))(x_1^2+\ldots x_n^2)^{s-d} \nonumber \\
&=\Psi_{s,x}(F)(x_1^2+\ldots+x_n^2)^{s-d}. \nonumber
\end{align}
We now show that (\ref{eq:Phi.s.x.inv}) holds: $$\Psi_{s,x}(\Psi_{s,x}^{-1}(F))=\Psi_{s,x} \left( \sum_{i,j}y_iy_j \Phi_{s}^{-1}(p_{ij}(x))\right)=\sum_{i,j}y_iy_j \Phi_s \Phi_{s}^{-1}(p_{ij})=\sum_{i,j} y_iy_j p_{ij}=F.$$
\end{proof}

\begin{lemma} \label{lem:Phi.psd}
	For a biform $F(x;y)$ of the structure in (\ref{eq:def.F}), which is positive on the bisphere, let $$\eta(F)\mathrel{\mathop{:}}=\frac{\min\{F(x;y)~|~(x,y) \in S^{n-1} \times S^{n-1}\}}{\max\{F(x;y)~|~(x,y) \in S^{n-1} \times S^{n-1}\}}.$$
	If $s \geq \frac{nd(d-1)}{4\log(2)\eta(F)}-\frac{n-d}{2}$, then $\Psi_{s,x}^{-1}(F)$ is positive semidefinite.
\end{lemma}

\begin{proof}
Fix $y \in S^{n-1}$ and consider $F_y(x)=F(x;y)$, which is a positive definite form in $x$ of degree $d$. From Proposition \ref{prop:pd.Phi}, if $$s\geq \frac{nd(d-1)}{4\log(2)\epsilon(F_y)}-\frac{n-d}{2},$$ then $\Phi^{-1}_s(F_y)$ is positive semidefinite. As $\eta(F) \leq \epsilon(F_{y})$ for any $y \in S^{n-1}$, we have that if $$s \geq \frac{nd(d-1)}{4\log(2)\eta(F)}-\frac{n-d}{2},$$ then $\Phi_s^{-1}(F_y)$ is positive semidefinite, regardless of the choice of $y.$ Hence, $\Psi_{s,x}^{-1}(F)$ is positive semidefinite (as a function of $x$ and $y$).

\end{proof}

\begin{proof}[Proof of Theorem \ref{th:r.sos.convex}]
Let $F(x;y)=y^TH_f(x)y$, let $r\geq \frac{n(d-2)(d-3)}{4\log(2)\eta(f)}-\frac{n+d-2}{2}-d$, and let $$G(x;y)=\Psi_{r+d,x}^{-1}(F).$$ We know by Lemma \ref{lem:Phi.psd} that $G(x;y)$ is positive semidefinite. Hence, using Lemma \ref{lem:F.sos}, we get that $$G(D,y)(x_1^2+\ldots+x_n^2)^{r+d}$$ is sos. Lemma \ref{lem:phi.def} then gives us:
	\begin{align*}
	G(D;y)(x_1^2+\ldots+x_n^2)^{r+d}&\underset{\text{using } (\ref{eq:def.Phi.s.x})}{=}\Psi_{r+d,x}(G)
	(x_1^2+\ldots+x_n^2)^r\\
	&\underset{\text{using } (\ref{eq:Phi.s.x.inv})}{=}F(x;y)(x_1^2+\ldots+x_n^2)^r.
	\end{align*}
	
As a consequence, $F(x;y)(x_1^2+\ldots+x_n^2)^r$ is sos.
	
\end{proof}
The last theorem of this section shows that one cannot bound the integer $r$ in Theorem \ref{th:r.sos.convex} as a function of $n$ and $d$ only.
\begin{theorem}\label{th:unif.bnd}
For any integer $r \geq 0$, there exists a form $f$ in 3 variables and of degree 8 such that $H_f(x) \succ 0, \forall x \neq 0$, but $f$ is not $r$-sos-convex.
\end{theorem}
\begin{proof}
Consider the trivariate octic:
	\begin{align*}
	f(x_1,x_2,x_3)&= 32x_1^8+118x_1^6x_2^2+40x_1^6x_3^2+25x_1^4x_2^2x_3^2-35x_1^4x_3^4+3x_1^2x_2^4x_3^2-16x_1^2x_2^2x_3^4+24x_1^2x_3^6\\
	&+16x_2^8+44x_2^6x_3^2+70x_2^4x_3^4+60x_2^2x_3^6+30x_3^8.
	\end{align*}
	It is shown in \cite{AAA_PP_not_sos_convex_journal} that $f$ has positive definite Hessian, and that the $(1,1)$ entry of $H_f(x)$, which we will denote by $H_f^{(1,1)}(x)$, is 1-sos but not sos. We will show that for any $r \in \mathbb{N}$, one can find $s \in \mathbb{N} \backslash \{0\}$ such that $$g_s(x_1,x_2,x_3)=f(x_1,sx_2,sx_3)$$ satisfies the conditions of the theorem.
	
	We start by showing that for any $s$, $g_s$ has positive definite Hessian. To see this, note that for any $(x_1,x_2,x_3) \neq 0, (y_1,y_2,y_3) \neq 0$, we have:
	$$(y_1, y_2,y_3) H_{g_s}(x_1,x_2,x_3)(y_1,y_2,y_3)^T=(y_1,sy_2,sy_3) H_f(x_1,sx_2,sx_3)(y_1,sy_2,sy_3)^T.$$
	As $y^TH_f(x)y>0$ for any $x \neq 0, y\neq 0$, this is in particular true when $x=(x_1, sx_2,sx_3)$ and when $y=(y_1, sy_2,sy_3)$, which gives us that the Hessian of $g_s$ is positive definite for any $s \in \mathbb{N} \backslash \{0\}.$

We now show that for a given $r\in \mathbb{N}$, there exists $s \in \mathbb{N}$ such that $(x_1^2+x_2^2+x_3^2)^r y^TH_{g_s}(x)y$ is not sos. We use the following result from \cite[Theorem 1]{Reznick1}: for any positive semidefinite form $p$ which is not sos, and any $r\in \mathbb{N}$, there exists $s \in \mathbb{N} \backslash \{0\}$ such that $(\sum_{i=1}^n x_i^2)^r\cdot p(x_1,sx_2,\ldots,sx_n)$ is not sos. As $H_f^{(1,1)}(x)$ is 1-sos but not sos, we can apply the previous result. Hence, there exists a positive integer $s$ such that $$(x_1^2+x_2^2+x_3^2)^r \cdot H_f^{(1,1)}(x_1,sx_2,sx_3)=(x_1^2+x_2^2+x_3^2)^r \cdot H_{g_s}^{(1,1)}(x_1,x_2,x_3)$$
is not sos. This implies that $(x_1^2+x_2^2+x_3^2)^r \cdot y^TH_{g_s}(x)y$ is not sos. Indeed, if $(x_1^2+x_2^2+x_3^2)^r \cdot y^TH_{g_s}(x)y$ was sos, then $(x_1^2+x_2^2+x_3^2)^r \cdot y^TH_{g_s}(x)y$ would be sos with $y=(1,0,0)^T.$ But, we have $$(x_1^2+x_2^2+x_3^2)^r \cdot (1,0,0)H_{g_s}(x)(1,0,0)^T=(x_1^2+x_2^2+x_3^2)^r \cdot H_{g_s}^{(1,1)}(x),$$
which is not sos. Hence, $(x_1^2+x_2^2+x_3^2)^r \cdot y^TH_{g_s}(x)y$ is not sos, and $g$ is not $r$-sos-convex.
\end{proof}

\begin{remark}\label{rem:pd.hess.str.conv}
Any form $f$ with $H_f(x) \succ 0, \forall x \neq 0$ is strictly convex but the converse is not true.

To see this, note that any form $f$ of degree $d$ with a positive definite Hessian is convex (as $H_f(x)\succeq 0, \forall x$) and positive definite (as, from a recursive application of Euler's theorem on homogeneous functions, $f(x)=\frac{1}{d(d-1)}x^TH_f(x)x$). From the proof of Theorem \ref{th:norm.str.conv}, this implies that $f$ is strictly convex.

To see that the converse statement is not true, consider the strictly convex form $f(x_1,x_2)\mathrel{\mathop{:}}=x_1^4+x_2^4$. We have $$H_f(x)=12\cdot \begin{bmatrix} x_1^2 & 0 \\ 0 & x_2^2 \end{bmatrix}$$ which is not positive definite e.g., when $x=(1,0)^T$.
\end{remark}

\subsubsection{Optimizing over a subset of polynomial norms with $r$-sos-convexity}

In the following theorem, we give a semidefinite programming-based hierarchy for optimizing over the set of forms $f$ with $H_f(x)\succ 0$, $\forall x \neq 0.$ Comparatively to Theorem \ref{th:test.poly.norm}, this theorem allows us to impose as a constraint that the $d^{th}$ root of a form be a norm, rather than simply testing whether it is. This comes at a cost however: in view of Remark \ref{rem:pd.hess.str.conv} and Theorem \ref{th:norm.str.conv}, we are no longer considering all polynomial norms, but a subset of them whose $d^{th}$ power has a positive definite Hessian.

\begin{theorem}\label{th:opt.poly.norms}
	Let $f$ be a degree-$d$ form. Then $H_f(x) \succ 0, \forall x\neq 0$ if and only if $\exists c>0, r \in \mathbb{N}$ such that $f(x)-c(\sum_i x_i^2)^{d/2}$ is $r$-sos-convex. Furthermore, this condition can be imposed using semidefinite programming.
\end{theorem}

\begin{proof}
	
	If there exist $c>0, r\in \mathbb{N}$ such that $g(x)=f(x)-c(\sum_i x_i^2)^{d/2}$ is $r$-sos-convex, then $y^TH_g(x)y \geq 0$, $\forall x,y.$ As the Hessian of $(\sum_i x_i^2)^{d/2}$ is positive definite for any nonzero $x$ and as $c>0$, we get $H_f(x)\succ 0$, $\forall x\neq 0.$

	Conversely, if $H_f(x)\succ 0$, $\forall x\neq 0$, then $y^TH_f(x)y>0$ on the bisphere (and conversely). Let $$f_{\min}\mathrel{\mathop{:}}= \min_{||x||=||y||=1} y^TH_f(x)y.$$ We know that $f_{\min}$ is attained and is positive. Take $c\mathrel{\mathop{:}}=\frac{f_{\min}}{2d(d-1)}$ and consider $$g(x)\mathrel{\mathop{:}}= f(x)-c(\sum_i x_i^2)^{d/2}.$$
	Then $$y^TH_g(x)y=y^TH_f(x)y-c \cdot \left( d(d-2)(\sum_i x_i^2)^{d/2-2} (\sum_i x_iy_i)^2 + d\sum_i (x_i^2)^{d/2-1} (\sum_i y_i^2) \right).$$
	Note that, by Cauchy-Schwarz, we have $(\sum_i x_i y_i)^2 \leq ||x||^2||y||^2$. If $||x||=||y||=1$, we get $$y^TH_g(x)y\geq y^TH_f(x)y-c(d(d-1))>0.$$
	Hence, $H_g(x)\succ 0, \forall x\neq 0$ and there exists $r$ such that $g$ is $r$-sos-convex from Theorem \ref{th:r.sos.convex}.

	For fixed $r$, the condition that there be $c>0$ such that $f(x)-c(\sum_i x_i^2)^{d/2}$ is $r$-sos-convex can be imposed using semidefinite programming. This is done by searching for coefficients of a polynomial $f$ and a real number $c$ such that
	\begin{equation}\label{eq:SDP.opt.poly.norm}
	\begin{aligned}
	&y^TH_{f-c (\sum_i x_i^2)^{d/2}}y \cdot (\sum_i x_i^2)^r \text{ sos}\\
	&c\geq 0.
	\end{aligned}
	\end{equation}
	Note that both of these conditions can be imposed using semidefinite programming.
\end{proof}

\begin{remark}
	Note that we are not imposing $c>0$ in the above semidefinite program. As mentioned in Section~\ref{sec:test}, this is because in practice the solution returned by interior point solvers will be in the interior of the feasible set.
	
	In the special case where $f$ is completely free\footnote{This is the case of our two applications in Section \ref{sec:apps}.} (i.e., when there are no additional affine conditions on the coefficients of $f$), one can take $c \geq 1$ in (\ref{eq:SDP.opt.poly.norm}) instead of $c \geq 0$. Indeed, if there exists $c>0$, an integer $r$, and a polynomial $f$ such that $f -c(\sum_i x_i^2)^{d/2}$ is $r$-sos-convex, then $\frac1c f$ will be a solution to (\ref{eq:SDP.opt.poly.norm}) with $c \geq 1$ replacing $c \geq 0$.
\end{remark}

%\begin{openprob}
%	A set is set to be SDP-representable if it can be written as the projection of a Linear Matrix Inequality representation \cite{Helton_Nie_SDP_repres}. Is the 1-sublevel set of any polynomial norm SDP-representable?
%\end{openprob}
%
%It is clear that the 1-sublevel set of any polynomial norm where the Hessian of the $d^{th}$ power is positive definite is SDP-representable. In \cite{Helton_Nie_SDP_repres_2}, Helton and Nie further show that any set with nonsingular positively curved boundary has an SDP representation. This is a weaker statement than requiring that the set be strictly convex (which corresponds to the case at hand) and to our knowledge, the previous problem remains open.

%======================================================
\section{Applications} \label{sec:apps}
%===================================================

\subsection{Norm approximation and regression}\label{sec:norm.reg}

In this section, we study the problem of approximating a (non-polynomial) norm by a polynomial norm. We consider two different types of norms: $p$-norms with $p$ noneven (and greater than 1) and gauge norms with a polytopic unit ball. For $p$-norms, we use as an example $||(x_1,x_2)^T||=(|x_1|^{7.5}+|x_2|^{7.5})^{1/7.5}$. For our polytopic gauge norm, we randomly generate an origin-symmetric polytope and produce a norm whose 1-sublevel corresponds to that polytope. This allows us to determine the value of the norm at any other point by homogeneity (see \cite[Exercise 3.34]{BoydBook} for more information on gauge norms, i.e., norms defined by convex, full-dimensional, origin-symmetric sets). To obtain our approximations, we proceed in the same way in both cases. We first sample $N=200$ points $x_1,\ldots,x_N$ {\ghh uniformly at random} on the sphere $S^{n-1}$. We then solve the following optimization problem with $d$ fixed:
\begin{equation}\label{eq:opt.norm.reg}
\begin{aligned}
&\min_{f\in H_{2,d}} &&\sum_{i=1}^{N} (||x_i||^d-f(x_i))^2\\
&\text{s.t. } &&f \text{ sos-convex}.
\end{aligned}
\end{equation}
Problem (\ref{eq:opt.norm.reg}) can be written as a semidefinite program as the objective is a convex quadratic in the coefficients of $f$ and the constraint has a semidefinite representation as discussed in Section \ref{sec:sos.review}. The solution $f$ returned is guaranteed to be convex. Moreover, any sos-convex form is sos (see \cite[Lemma 8]{Helton_Nie_SDP_repres}), which implies that $f$ is nonnegative. One can numerically check to see if the optimal polynomial is in fact positive definite (for example, by checking the eigenvalues of the Gram matrix of a sum of squares decomposition of $f$). If that is the case, then, by Theorem \ref{th:norm.str.conv}, $f^{1/d}$ is a norm. Futhermore, note that we have
\begin{align*}
\left(\sum_{i=1}^N (||x_i||^d-f(x_i))^2 \right)^{1/d} &\geq \frac{N^{1/d}}{N} \sum_{i=1}^N (||x_i||^d-f(x_i))^{2/d}\\
&\geq \frac{N^{1/d}}{N} \sum_{i=1}^N (||x_i||-f^{1/d}(x_i))^2,
\end{align*}
where the first inequality is a consequence of concavity of $z \mapsto z^{1/d}$ and the second is a consequence of the inequality $|x-y|^{1/d} \geq ||x|^{1/d}-|y|^{1/d}|$. This implies that if the optimal value of (\ref{eq:opt.norm.reg}) is equal to $\epsilon$, then the sum of the squared differences between $||x_i||$ and $f^{1/d}(x_i)$ over the sample is less than or equal to $N \cdot (\frac{\epsilon}{N})^{1/d}$.

It is worth noting that in our example, we are actually searching over the entire space of polynomial norms of a given degree. Indeed, as $f$ is bivariate, it is convex if and only if it is sos-convex \cite{AAA_PP_table_sos-convexity}. In Figure \ref{fig:norm.approx}, we have drawn the 1-level sets of the initial norm (either the $p$-norm or the polytopic gauge norm) and the optimal polynomial norm obtained via (\ref{eq:opt.norm.reg}) with varying degrees $d$. Note that when $d$ increases, the approximation improves.

\begin{figure}[h]
	\centering
	\begin{subfigure}[b]{0.49\textwidth}
		\includegraphics[width=\textwidth]{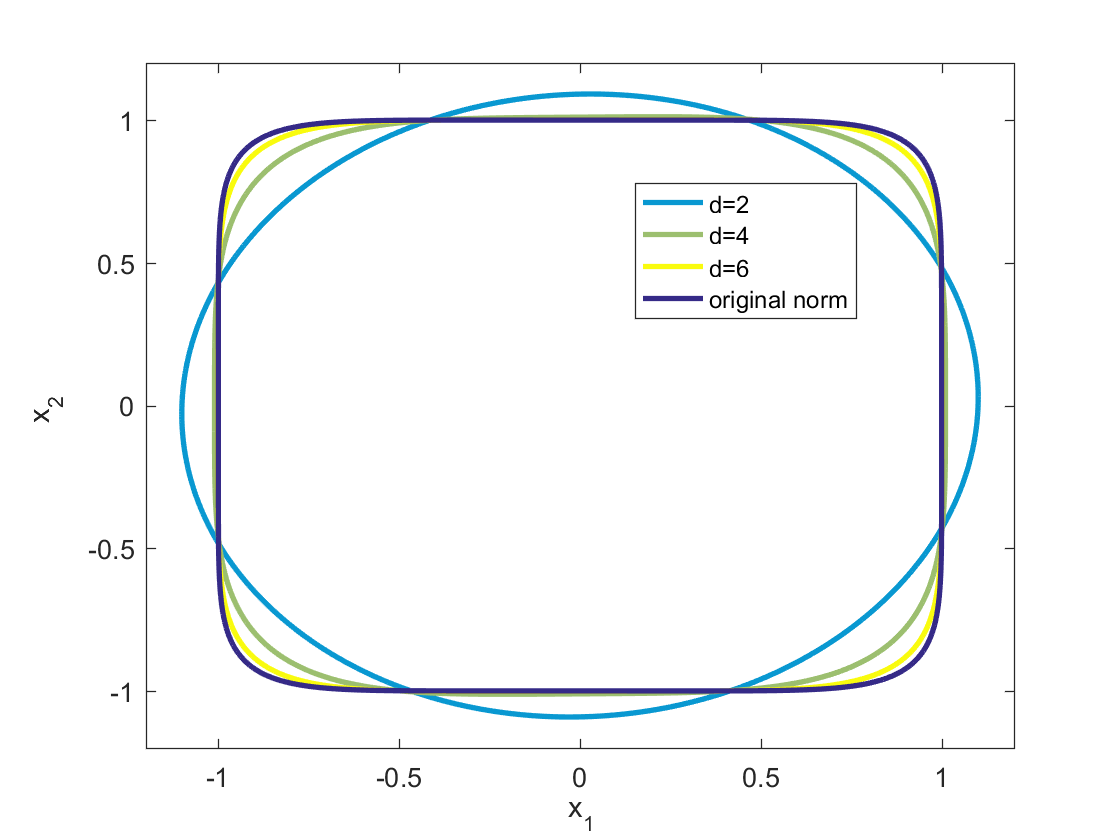}
		\caption{p-norm approximation}
	\end{subfigure}
	~ %add desired spacing between images, e. g. ~, \quad, \qquad, \hfill etc.
	\begin{subfigure}[b]{0.49\textwidth}
		\includegraphics[width=\textwidth]{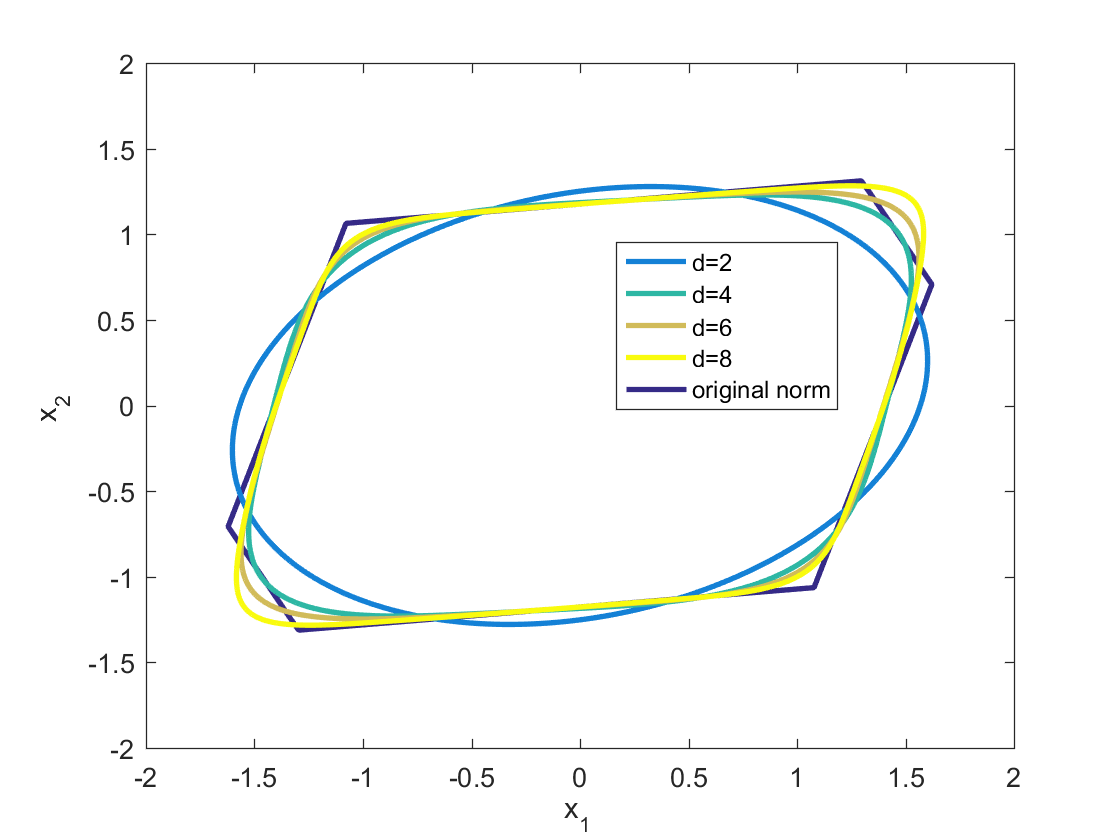}
		\caption{Polytopic norm approximation}
	\end{subfigure}
	\caption{Approximation of non-polynomial norms by polynomial norms}\label{fig:norm.approx}
\end{figure}

A similar method could be used for \emph{norm regression}. In this case, we would have access to data points $x_1,\ldots,x_N$ corresponding to noisy measurements of an underlying unknown norm function. We would then solve the same optimization problem as the one given in (\ref{eq:opt.norm.reg}) to obtain a polynomial norm that most closely approximates the noisy data.

\subsection{Joint spectral radius and stability of linear switched systems}\label{sec:JSR.comp}

As a second application, we revisit a result from one of the authors and Jungers from \cite{sosconvex_Lyap_cdc,Ahmadi_Jungers} on finding {\rev upper bounds} on the joint spectral radius of a finite set of matrices. We first review a few notions relating to dynamical systems and linear algebra.
The spectral radius $\rho$ of a matrix $A$ is defined as $$\rho(A)=\lim_{k \rightarrow \infty} ||A^k||^{1/k}.$$ The spectral radius happens to coincide with the eigenvalue of $A$ of largest magnitude. Consider now the discrete-time linear system $x_{k+1}=Ax_k$, where $x_k$ is the $n \times 1$ state vector of the system at time $k$. This system is said to be \emph{asymptotically stable} if for any initial starting state $x_0 \in \mathbb{R}^n$, $x_k \rightarrow 0,$ when $k \rightarrow \infty.$ A well-known result connecting the spectral radius of a matrix to the stability of a linear system states that the system $x_{k+1}=Ax_k$ is asymptotically stable if and only if $\rho(A)<1$.

In 1960, Rota and Strang introduced a generalization of the spectral radius to a \emph{set} of matrices. The \emph{joint spectral radius (JSR)} of a set of matrices $\mathcal{A} \mathrel{\mathop{:}}=\{A_1,\ldots,A_m\}$ is defined as
\begin{align}\label{eq:JSR.def}
\rho(\mathcal{A})\mathrel{\mathop{:}}=\lim_{k\rightarrow \infty} \max_{\sigma \in \{1,\ldots,m\}^k}||A_{\sigma_k} \ldots A_{\sigma_1}||^{1/k}.
\end{align}
Analogously to the case where we have just one matrix, the value of the joint spectral radius can be used to determine stability of a certain type of system, called \emph{a switched linear system.} A switched linear system models an uncertain and time-varying linear system, i.e., a system described by the dynamics
\begin{align}\label{eq:switch}
x_{k+1}=A_{k}x_k,
\end{align}
where the matrix $A_k$ varies at each iteration within the set $\mathcal{A}$. As done previously, we say that a switched linear system is asymptotically stable if {\rev $x_k \rightarrow 0$} when $k \rightarrow \infty$, for any starting state $x_0 \in \mathbb{R}^n$ and any sequence of products of matrices in $\mathcal{A}$. One can establish that the switched linear system $x_{k+1}=A_{k}x_k$ is {\rev asymptotically} stable if and only if $\rho(\mathcal{A})<1$ \cite{Raphael_Book}.

{\rev Switched linear systems are typically used to model situations where the dynamics of a system are thought to be linear, but the matrix $A_k \in \mathcal{A}$ associated to the linear dynamics $x_{k+1}=A_kx_k$ is unknown and time-varying. Consider, e.g., the task of stabilizing a drone in a windy environment. By linearizing its dynamics around a desired equilibrium point, the behavior of the drone can be modeled locally by a linear dynamical system. However, as this linear dynamical system is unknown due to parameter uncertainty and modeling error, and time-varying due to the effect of the wind, the drone's behavior is better modeled by a switched linear system. 
		
Consequently, a natural question is whether one can efficiently test if $\rho(\mathcal{A})<1$ and hence determine if the corresponding switched linear system is asymptotically stable. Unlike the setting of linear systems, where one can decide whether the spectral radius of a matrix is less than one in polynomial time, it is not known whether the problem of testing if $\rho(\mathcal{A})<1$ is even decidable. The related question of testing whether $\rho(\mathcal{A})\leq 1$ is known to be undecidable, already when $\mathcal{A}$ contains only 2 matrices \cite{BlTi2}. With this result in mind, it comes as no surprise that, e.g., stability of a switched linear system is not implied by all individual matrices in $\mathcal{A}$ having spectral radius less than one. This is easy to see on an example: consider the set of matrices $\mathcal{A}$ given by
	$$A_1=\begin{bmatrix} 0 & 2 \\ 0 & 0\end{bmatrix} \text{ and } A_2=\begin{bmatrix} 0 & 0 \\ 2 & 0 \end{bmatrix}.$$
Observe that the spectral radii of $A_1$ and $A_2$ are zero, which is less than one. However $$A_1A_2=\begin{bmatrix} 4 & 0 \\ 0 & 0 \end{bmatrix}$$ and so $\rho(\mathcal{A})$ is lower bounded by $2>1$, and the switched linear system is not stable. }
%
%	
%Hence, to determine whether a switched linear system is asymptotically stable, it is of interest to compute the joint spectral radius of $\mathcal{A}$. However, unlike the spectral radius which can be computed in polynomial time (and so compared to 1 in polynomial time), the problem of testing whether $\rho(\mathcal{A})\leq 1$ is undecidable, already when $\mathcal{A}$ contains only 2 matrices \cite{BlTi2}.  }
%%
%
%
%Though they may seem similar on many points, a key difference between the spectral radius and the joint spectral radius lies in difficulty of computation: testing whether the spectral radius of {\rev a matrix $A$ is less than or equal to $1$ (or strictly less than 1) can be done in polynomial time.}  However, already when $m=2$, the problem of testing whether $\rho(A_1,A_2)\leq 1$ is undecidable \cite{BlTi2}.

An active area of research has consequently been to obtain sufficient conditions for the JSR to be strictly less than one, which, for example, can be checked using convex optimization. The theorem that we revisit below is a result of this type. We start first by recalling a classical theorem regarding stability of a linear system.
\begin{theorem}[see, e.g., Theorem 8.4 in \cite{hespanha2009linear}]\label{th:AAA.Jungers}
	Let $A \in \mathbb{R}^{n \times n}$. Then, $\rho(A)<1$ if and only if there exists a contracting quadratic norm; i.e., a function $V:\mathbb{R}^n \rightarrow \mathbb{R}$ of the form $V(x)=\sqrt{x^TQx}$ with $Q\succ 0$, such that $V(Ax)<V(x), \forall x\neq 0.$
\end{theorem}

The next theorem (from \cite{sosconvex_Lyap_cdc,Ahmadi_Jungers}) can be viewed as an extension of Theorem \ref{th:AAA.Jungers} to the joint spectral radius of a finite set of matrices. It is known that the existence of a contracting quadratic norm is no longer necessary for stability in this case. This theorem {\rev shows} however that the existence of a contracting polynomial norm is.

\begin{theorem}[adapted from \cite{sosconvex_Lyap_cdc,Ahmadi_Jungers}, Theorem 3.2 ]\label{th:poly.norms.JSR}
	Let $\mathcal{A}\mathrel{\mathop{:}}=\{A_1,\ldots,A_m\}$ be a family of $n \times n$ matrices. Then, $\rho(A_1,\ldots,A_m)<1$ if and only if there exists a contracting polynomial norm; i.e., a function $V(x)=f^{1/d}(x)$, where $f$ is an n-variate sos-convex and positive definite form of degree $d$, such that $V(A_ix)<V(x),~\forall x\neq 0$ and $\forall i=1,\ldots,m.$
\end{theorem}

We remark that in \cite{ahmadi2016lower}, the authors show that the degree of $f$ cannot be bounded as a function of $m$ and $n$. This is expected from the undecidability result mentioned before.

\begin{example}
	We consider a modification of Example 5.4. in \cite{ahmadi2011analysis} as an illustration of the previous theorem. We would like to show that the joint spectral radius of the two matrices $$A_1=\frac{1}{3.924}\begin{bmatrix} -1 & -1 \\ 4 & 0 \end{bmatrix}, \quad A_2=\frac{1}{3.924}\begin{bmatrix} 3 & 3 \\ -2 & 1 \end{bmatrix}$$ is strictly less that one.
	
	To do this, we search for a nonzero form $f$ of degree $d$ such that
	\begin{equation} \label{eq:JSR.opt.prob}
	\begin{aligned}
	&f -(\sum_{i=1}^n x_i^2)^{d/2}\text{ sos-convex}\\
	&f(x)-f(A_ix)- (\sum_{i=1}^n x_i^2)^{d/2} \text{ sos}, \text{ for } i=1,2.
	\end{aligned}
	\end{equation}
	If problem (\ref{eq:JSR.opt.prob}) is feasible for some $d$, then $\rho(A_1,A_2)<1$. A quick computation using the software package YALMIP \cite{yalmip} and the SDP solver MOSEK \cite{mosek} reveals that, when $d=2$ or $d=4$, problem (\ref{eq:JSR.opt.prob}) is infeasible. When $d=6$ however, the problem is feasible and we obtain a polynomial norm $V=f^{1/d}$ whose 1-sublevel set is the outer set plotted in Figure \ref{fig:level.sets}. We also plot on Figure \ref{fig:level.sets} the images of this 1-sublevel set under $A_1$ and $A_2$. Note that both sets are included in the 1-sublevel set of $V$ as expected. From Theorem \ref{th:poly.norms.JSR}, the existence of a polynomial norm implies that $\rho(A_1,A_2)<1$ and hence, the pair $\{A_1,A_2\}$ is asymptotically stable.

	\begin{figure}[h]
		\centering
		\includegraphics[scale=0.3]{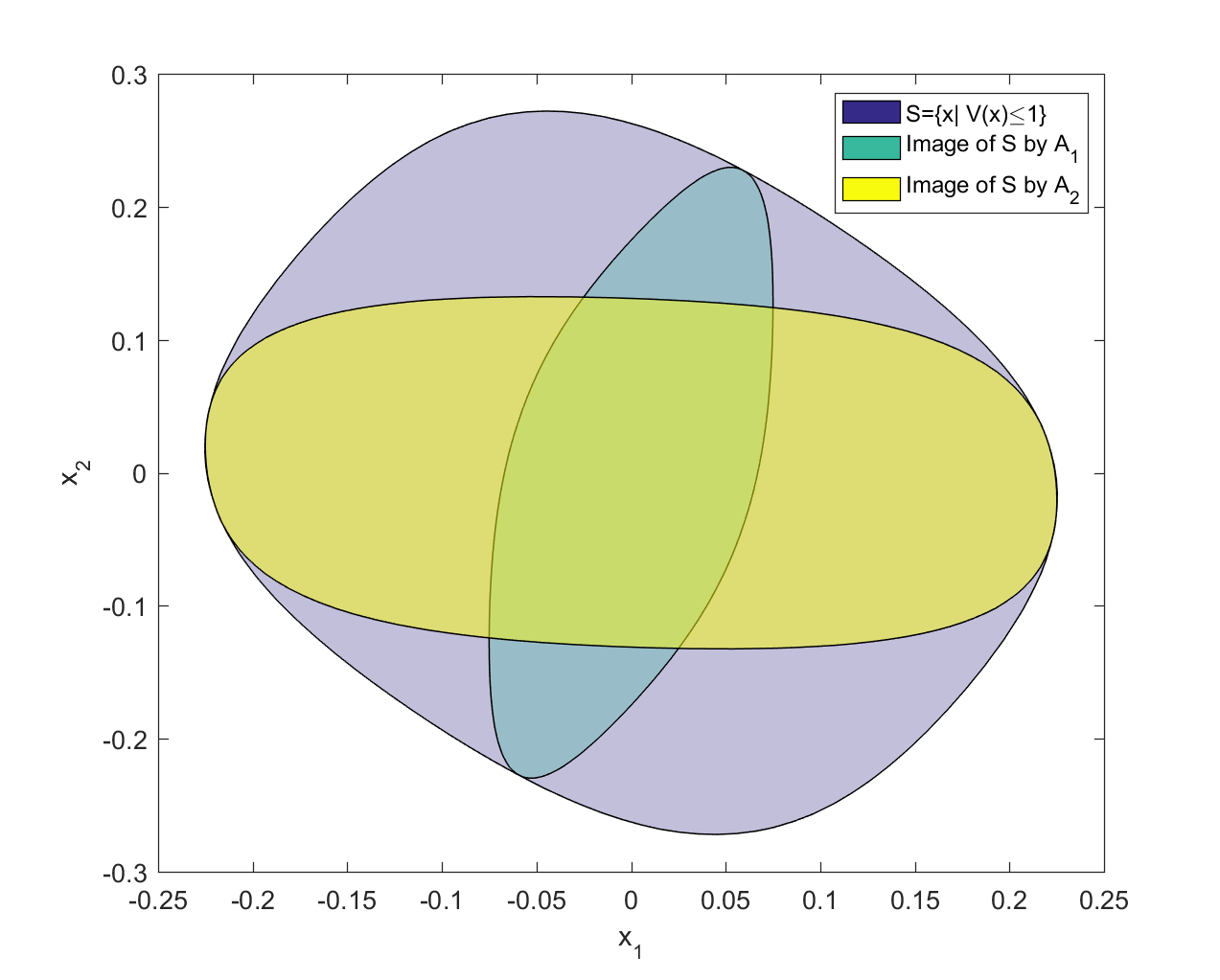}
		\caption{Image of the sublevel set of $V$ under $A_1$ and $A_2$}
		\label{fig:level.sets}
	\end{figure}

\end{example}

\begin{remark}
	As mentioned previously, problem (\ref{eq:JSR.opt.prob}) is infeasible for $d=4$. Instead of pushing the degree of $f$ up to 6, one could wonder whether the problem would have been feasible if we had asked that $f$ of degree $d=4$ be $r$-sos-convex for some fixed $r \geq 1$. As mentioned before, in the particular case where $n=2$ (which is the case at hand here), the notions of convexity and sos-convexity coincide; see \cite{AAA_PP_table_sos-convexity}. As a consequence, one can only hope to make problem (\ref{eq:JSR.opt.prob}) feasible by increasing the degree of $f$.
\end{remark}

%==============================================
\section{Future directions}\label{sec:open.problem}
%=============================================

{\ghh In this paper, we provided semidefinite programming-based hierarchies for certifying that the $d^{th}$ root of a given degree-$d$ form is a polynomial norm (Section \ref{sec:test}), and for optimizing over the set of forms with positive definite Hessians (Section \ref{sec:opt}).} A clear gap emerged between forms which are strictly convex and those which have a positive definite Hessian, the latter being a sufficient (but not necessary) condition for the former. This leads us to consider the following two open problems.

\begin{openprob}
	Does there exist a family of cones $K^{r}_{n,2d}$ that have the following two properties: (i) for each $r$, optimization of a linear function over $K_{n,2d}^r$ can be carried out with semidefinite programming, and (ii) every strictly convex form $f$ in $n$ variables and degree $2d$ belongs to $K_{n,2d}^r$ for some $r$? We have shown a weaker result, namely the existence of a family of cones that verify (i) and a modified version of (ii), where strictly convex forms are replaced by forms with a positive definite Hessian.

%We have given a {\ghh semidefinite programming hierarchy\footnote{By a semidefinite programming hierarcy, we mean a family of cones $K^{r}_{n,2d}$ that have two properties: for each $r$, optimization over $K_{n,2d}^r$ can be done with semidefinite programming, and (ii) every strictly convex form $f$ in $n$ variables and degree $2d$ should belong to $K_{n,2d}^r$ for some $r$.} for optimizing over a subset of polynomial norms.} Is there a semidefinite programming hierarchy that enables us to optimizes over all polynomial norms?
\end{openprob}

\begin{openprob}
Helton and Nie have shown in \cite{Helton_Nie_SDP_repres} that one can optimize a linear function over sublevel sets of forms that have positive definite Hessians with semidefinite programming. Is the same statement true for sublevel sets of all polynomial norms?
\end{openprob}

%Another open problem that we think would be of interest relates to the quality of approximation of norms by polynomial norms.
%
%\begin{openprob}
%	 As mentioned previously, Barvinok has shown in \cite{barvinok2003} that for any norm $||.||$, there exists a nonnegative form $f$ of degree $2d$ such that
%	$$f^{1/2d}(x) \leq ||x|| \leq \binom{n+d-1}{d}^{1/2d} f^{1/2d}(x).$$ We have shown in Section \ref{sec:approx.norms} that for any $\epsilon>0$, there exists an sos-convex and positive definite polynomial $f$ of degree $d$ such that $$f^{1/2d}(x) \leq ||x|| \leq (1+\epsilon) f^{1/2d}(x).$$ Is it possible to quantify the degree $d$ needed to obtain an approximation of precision $\epsilon$ as a function of $\epsilon$ and $n$ only? How would the degree be impacted if we searched for an $r$-sos-convex polynomial instead? In this case, could we express $d$ as a function of $\epsilon$, $n$, and $r$ only?
%\end{openprob}
On the application side, it might be interesting to investigate how one can use polynomial norms to design \emph{regularizers} in machine learning applications. Indeed, a very popular use of norms in optimization is as regularizers, with the goal of imposing additional structure (e.g., sparsity or {\rev low rank}) on optimal solutions. One could imagine using polynomial norms to design regularizers that are based on the data at hand in place of more generic regularizers such as the 1-norm. Regularizer design is a problem that has already been considered (see, e.g., \cite{bach2012optimization,venkat}) but not using polynomial norms. This can be worth exploring as we have shown that polynomial norms can approximate any norm with arbitrary accuracy, while remaining differentiable everywhere (except at the origin), which can be beneficial for optimization purposes.

\subsubsection*{Acknowledgement}
{\ghh The authors would like to thank an anonymous referee for suggesting the
proof of the first part of Theorem \ref{th:approx.poly.norm.sphere}, which improves our previous statement by quantifying the quality of the approximation as a function of $n$ and $d$, and two other anonymous referees for constructive comments that considerably helped improve the draft.}

\bibliographystyle{amsplain}
\bibliography{pablo_amirali}

\end{document}